\RequirePackage{fix-cm}

\documentclass[smallextended]{svjour3}                   
\smartqed  
\usepackage{graphicx}

\usepackage{amssymb,amsmath}
\usepackage{amsmath}
\usepackage{amsfonts}

\usepackage{amssymb}
\usepackage{amsfonts}
\usepackage{latexsym}

\usepackage{subfigure}
\usepackage{color}

\usepackage{setspace}

\newcommand{\oo}{\mbox{$\mathbb O$}}

\newcommand{\x}{\mbox{$\mathbf x$}}
\newcommand{\y}{\mbox{$\mathbf y$}}
\newcommand{\vv}{\mbox{$\mathbf v$}}
\newcommand{\w}{\mbox{$\mathbf w$}}
\newcommand{\zz}{\mbox{$\mathbf z$}}

\newcommand{\ttt}{\mbox{$\cal T$}}
\newcommand{\f}{\mbox{$\cal F$}}
\newcommand{\F}{\mbox{$\cal F$}}

\newcommand{\sss}{\mbox{$\cal S$}}

\newcommand{\qa}{\mbox{\quad\mbox{and}\quad}}

\newcommand{\rt}{\mbox{$\mathbb R$}}

 \newcommand{\rank}{\mathrm{rank\;}}

\newcommand{\balpha}{\mbox{\boldmath $\alpha$}}

\newcommand{\bgamma}{\mbox{\boldmath $\gamma$}}

 \def\diag{\mathop{{\rm diag}}\nolimits}

\usepackage[compress]{cite}

\begin{document}

\title{Best  approximations of non-linear mappings:  Method of optimal injections}


\author{ Anatoli~Torokhti  \and Pablo~Soto-Quiros     
}


\institute{ Anatoli~Torokhti \at
              CIAM, University of South Australia, SA 5095, Australia \\
              Tel.: +61-8-83023812\\
              \email{anatoli.torokhti@unisa.edu.au}           
           \and
          Pablo~Soto-Quiros    \at
             Instituto Tecnologico de Costa Rica, Apdo. 159-7050, Cartago, Costa Rica\\
              Tel.: +506-25502025\\
              \email{jusoto@tec.ac.cr}           
}

\date{Received: date / Accepted: date}

\maketitle

\begin{abstract}
 While the theory of operator approximation with any given accuracy is  well elaborated, the theory of {best constrained} constructive operator
approximation  is still not so well developed.  Despite increasing demands from applications this subject is hardly tractable
because of intrinsic difficulties in associated  approximation techniques. This paper concerns the best constrained approximation of a non-linear operator  in probability spaces. We propose and justify a new approach and technique based on the following observation. Methods for best approximation are aimed at obtaining the best solution within a certain class; the accuracy of the solution is limited by the extent to which the class is suitable.
 By contrast, iterative methods are normally convergent but the convergence can be  quite slow.
Moreover, in practice only a finite number of iteration loops can  be carried out and therefore, the final approximate solution is often unsatisfactorily inaccurate.
A natural idea  is to combine the methods for  best approximation and iterative techniques to exploit their advantageous
features.  Here, we present an approach which realizes this.

The proposed approximating operator  has several degrees of freedom to minimize the associated error.
In particular, one of the specific features of the  approximating technique we develop is special random vectors  called injections. They are determined in the way that allows us to further minimize the associated error.

\end{abstract}

\vspace*{3mm}

\section{Introduction}\label{intro}

Let $\Omega$ be a set of outcomes in probability space $(\Omega, \Sigma, \mu)$ for which $\Sigma$ is a $\sigma$--field of measurable subsets of $\Omega$ and $\mu:\Sigma \rightarrow [0,1]$ is an associated probability measure.
 Let $\x\in  L^2(\Omega,\mathbb{R}^m)$ and $\y\in  L^2(\Omega,\mathbb{R}^n)$  be random vectors, $\f:L^2(\Omega,\mathbb{R}^n)\rightarrow L^2(\Omega,\mathbb{R}^m)$ be a   non-linear map  and $\x = \f(\y)$. It is assumed that $\f$ is unknown, and $\x$ and $\y$ are available.

We propose and justify a new method for the constructive $\f$ approximation such that first, an associated error is minimized and second, a structure of the approximating operator  satisfies the special constraint related to the dimensionality reduction of vector $\y$.


\subsection{Motivation}\label{mot}

For the last  decades, the problem of  constructive approximation of non-linear operators has been a topic of profound research. A number of fundamental papers have appeared which established  significant advances in this research area.  Some relevant references can be
found, in particular, in \cite{Amat405}, \cite{Bruno13}, \cite{Chen911}, \cite{dingle17},   \cite{Gallman619},  \cite{Howlett353},  \cite{Istr118}, \cite{Prenter341}, \cite{Prolla36},  \cite{Rao726}, \cite{Sandberg863},
\cite{Sandberg40,Sandberg1586},  \cite{Timofte2005}.

The known related  results  mainly concern  proving the existence and uniqueness of operators approximating a given map, and for justifying the bounds of errors arising from the approximation methods. The assumptions are that preimages and images are deterministic and can be represented in an
analytical form, that is, by equations. At the same time, in many applications,  the sets of preimages and images   are stochastic and cannot be described by equations. Nevertheless, it is possible to represent these  sets  in terms of their numerical characteristics, such as the  expectation and  covariance matrices. Typical examples are  stochastic signal processing \cite{Tomasz1293,905856,Zhu745,Wang2015,7482957}, statistics \cite{Brillinger2001,Jolliffe2002,buhlmann2011,yiy763,tor5277}, engineering \cite{Saghri2010,Kountchev91,moura993,azimi169} and  image processing \cite{burge2006,Phophalia201424}; in the latter case, a digitized image, presented by a matrix, is often interpreted as the sample of a stochastic signal.

While the theory of operator approximation with any given accuracy
is  well elaborated   (see, e.g.,  \cite{Amat405}, \cite{Bruno13},\cite{Chen911},  \cite{dingle17}),  \cite{Gallman619},  \cite{Howlett353},  \cite{Istr118}, \cite{Prenter341}, \cite{Prolla36},  \cite{Rao726},
\cite{Sandberg40}, \cite{Sandberg863}, \cite{Sandberg1586}, \cite{Timofte2005}), the theory of {\em best constrained} constructive operator approximation  is still not so well developed, although this  is an area of intensive recent research (see, e.g.,
\cite{Chen1191,Chen319,Fomin1334,Temlyakov33,torok111,Wei2016}). Despite increasing demands from applications \cite{Aminzare91,Zhu745,Wang2015,7482957,Jolliffe2002,buhlmann2011,yiy763,Saghri2010,
Kountchev91,moura993,azimi169,Phophalia201424, Chen1191,Chen319,Fomin1334,Temlyakov33,Wei2016,Schneider295, Billings1013,Piroddi1767, Alter11828,Gianfelici2009422,Formaggia2009,Ambrosi2002,Tomasz3339,Poor2001} this subject is hardly tractable
because of intrinsic difficulties in best approximation techniques, especially when the approximating
operator should have a specific structure implied by
the underlying problem.

We wish to extend the known results in this area to the case when the sets of preimages and images of  map $\F$ are stochastic, and the approximating operator we search is constructive in the sense it can numerically be realized and, therefore, is applicable to  problems in applications.

\subsection{Short description of the method}\label{}

More specifically, we develop   a new approach to the best constructive approximation of a non-linear
operator $\f$ in probability spaces subject to a specialized  criterion associated  with the dimensionality reduction of random preimages. The latter constraint follows from the requirements  in applications such as those considered in  \cite{Brillinger2001,Jolliffe2002,buhlmann2011,yiy763,tor5277,Kountchev91,Alter11828,Gianfelici2009422}. In particular, in signal processing and
system theory, a dimensionality reduction of random signals is used to optimize the cost of signal transmission.
It is assumed that the only available  information on  $\F$ is  given by certain covariance matrices formed from the preimages and images. This is a typical assumption used in the applications such as those considered, e.g., in \cite{Tomasz1293,905856,Zhu745,Wang2015,7482957,Brillinger2001,Jolliffe2002,buhlmann2011,yiy763,Schneider295,Aminzare91,Ambrosi2002,Tomasz3339,Poor2001}. Here, we adopt that assumption. As mentioned, in particular, in 
\cite{Ledoit2004365,ledoit2012,Adamczak2009,Vershynin2012,won2013,yang1994},
{\em a priori} knowledge of the covariances can come either from specific data models, or, after sample estimation during a training phase.


 The problem we consider (see (\ref{gxd11}) below) concerns finding the best approximating operator that depends on $2p+2$ unknown matrices $G_j$ and $H_j$, for $j=0,\ldots,p$ and $p$ more unknown random vectors $\vv_1,\ldots, \vv_p$. We call $\vv_1,\ldots, \vv_p$ the {\em injections}. Here, $p$ is a non-negative integer $p$. The injections $\vv_1,\ldots, \vv_p$ are aimed to further diminish the associated error. The difficulty is that $3p+2$ unknowns should be determined from a minimization of the single cost function given in (\ref{gxd11}).

The solution is presented in Section \ref{n5o8}  and is based on the following observation. Methods for best approximation are aimed at obtaining the best solution within a certain class; the accuracy of the solution is limited by the extent to which the class is satisfactory.
 By contrast, iterative methods are normally convergent but the convergence can be  quite slow.
Moreover, in practice only a finite number of iteration loops can  be carried out and therefore, the final approximate solution is often unsatisfactorily inaccurate.
A natural idea  is to combine the methods for  best approximation and iterative techniques to exploit their advantageous
features.

In  Section \ref{n5o8}, we present an approach which realizes this. First,  a special iterative
procedure is proposed which aims to improve the accuracy of approximation with
each consequent iteration loop. Secondly, the best approximation problem is
solved providing the smallest associated error within the chosen class of
approximants for  each iteration loop.
In  Section \ref{hh58}, we show that the combination of these  techniques allows us to build a
computationally efficient and flexible method. In particular, we prove that  the
error in approximating $\F$ by the proposed method  decreases with an increase of
the number of iterations. An application is made to the optimal filtering of
stochastic signals.

\section{The proposed approach }\label{new}

 \subsection{Some special notation}\label{some}

Let us write  $\x=[\x_{(1)},\ldots, \x_{(m)}]^T$ and $\y=[\y_{(1)},\ldots, \y_{(n)}]^T$  where $\x_{(i)}, \y_{(j)}\in  L^2(\Omega,\mathbb{R})$, for $i=1,\ldots,m$ and $j=1,\ldots,n$,
and $\x(\omega)\in\rt^m$ and $\y(\omega)\in\rt^n$ for all $\omega\in \Omega$.

 Each matrix $A\in\rt^{m\times n}$ defines a bounded linear transformation ${\mathcal A}: L^2(\Omega,\mathbb{R}^n) \rightarrow L^2(\Omega,\mathbb{R}^m)$. It is customary to write $A$ rather then ${\mathcal A}$ since $[{\mathcal A}(\x)](\omega) = A[\x(\omega)]$, for each $\omega\in \Omega$.

Let us also denote
\begin{eqnarray}\label{nm37}
\|{\bf x}\|^2_\Omega  =\int_\Omega \sum_{j=1}^m [\x_j(\omega)]^2 d\mu(\omega) < \infty.
\end{eqnarray}
The covariance matrix formed from $\x$ and $\y$ is denoted by
\begin{eqnarray}\label{cb903}
E_{xy}=\left\{\int_\Omega {\bf x}_{(i)}(\omega){\bf y}_{(j)}(\omega)d\mu(\omega)\right\}_{i,j=1}^{m,n}.
\end{eqnarray}

The Moore-Penrose pseudo-inverse \cite{ben1974} of matrix $M$ is denoted by $M^{\dag}$.

 \subsection{Generic structure of approximating operator}\label{nmsmi}

Let $\vv_1,\ldots, \vv_p$ be random vectors such that $\vv_j\in L^2(\Omega,\mathbb{R}^{q_j})$, for $j=1,\ldots, p$. We write $\y=\vv_0$ and $q_0=n$. As  mentioned before, we call $\vv_1,\ldots, \vv_p$ the injections. This is because $\vv_1,\ldots, \vv_p$ contribute to the decrease of the associated error as shown in Section \ref{fk92} below.
The choice of  $\vv_1,\ldots, \vv_p$ is considered in Section \ref{jkw8} where each  $\vv_j$, for $j=1,\ldots,p$,  is defined by a non-linear transformation $\varphi_j$ of $\y$, i.e., $\vv_j= \varphi_j(\y)$. To facilitate the numerical implementation of the approximating technique  introduced below, each vector $\vv_j$, for $j=1,\ldots, p$, is transformed to vector $\zz_j\in  L^2(\Omega,\mathbb{R}^{q_j})$ by transformation $Q_j$ so that
\begin{eqnarray}\label{772x2}
\zz_j = Q_j(\vv_j, Z_{j-1}),
\end{eqnarray}
where $Z_{j-1}=\{\zz_0,\ldots, \zz_{j-1}\}$.
The choice of $Q_j$ is considered in Section \ref{kla7}.

Further, for $i=0,1,\ldots,p$, let  $ G_i\in\rt^{m\times r_i},$ $H_i\in\rt^{r_i\times q_i}$ where $r_i$  is given, $0<r_i <r$ and
\begin{eqnarray}\label{hhx2}
r=r_0 +\ldots +r_p.
\end{eqnarray}
 Here, $r$ is a positive integer such that $r\leq \min \{m, n\}$.     

It is convenient to set $Q_0=I$ and $\zz_0=\vv_0=\y$.  To approximate $\f$, for a  given  reduction ratio
 $$
 \displaystyle c=r/{\min \{m, n\}},
  $$
  we consider  operator $ \ttt_p: L^2(\Omega,\mathbb{R}^{q_0})\times \ldots \times L^2(\Omega,\mathbb{R}^{q_p})\rightarrow L^2(\Omega,\mathbb{R}^{m})$ represented by
\begin{eqnarray}\label{hq1q2x2}
 \ttt_p(\vv_0,\ldots\vv_p) = G_0 H_0\zz_0  +\ldots  + G_p H_p \zz_p,
\end{eqnarray}
where $G_j: L^2(\Omega,\mathbb{R}^{r_j})\rightarrow L^2(\Omega,\mathbb{R}^{m})$ and $H_j: L^2(\Omega,\mathbb{R}^{q_j})\rightarrow L^2(\Omega,\mathbb{R}^{r_j})$, for $j=0,1,\ldots,p$, are linear operators (i.e. $G_j$ and $H_j$  are represented by $m\times r_j$ and $r_j\times q_j$ matrices, respectively. Recall, we use the same symbol to define a matrix and the associated liner operator).

Importantly, operators $H_0,\ldots, H_p$ imply the dimensionality reduction of vectors $\vv_0,$ $\ldots, \vv_p$. This is because $H_i\zz_i\in  L^2(\Omega,\mathbb{R}^{r_i})$ where $0< r_i < r\leq \min \{m, n\}$, for $i=0,\ldots, p$.

We call $p$ the degree of $ \ttt_p$.
It is shown below that $\ttt_p$ approximates an operator of interest $\F:  L^2(\Omega,\mathbb{R}^{n})\rightarrow L^2(\Omega,\mathbb{R}^{m})$ with the accuracy represented by theorems in Sections \ref{err1},  \ref{yy41} and \ref{fk92}.

 \subsection{Statement of the problem}\label{nm109}

Let $\F:  L^2(\Omega,\mathbb{R}^{n})\rightarrow L^2(\Omega,\mathbb{R}^{m})$ be a continuous  operator.
We consider the problem as follows: Given $\x, \y$ and $r_0,\ldots,r_p$, find  matrices $G_0, H_0, \ldots, G_p, H_p$ and vectors $\vv_1,$ $\ldots, \vv_{p}$  that solve
\begin{eqnarray}\label{gxd11}
\min_{\mathbf{z}_1,\ldots, \mathbf{z}_p}\hspace*{2mm} \min_{\substack{G_0, H_0,\ldots, G_p, H_p\\}}  \left\|\f(\y) -  \sum_{j=0}^{p}G_j H_j \zz_j\right\|^2_\Omega
\end{eqnarray}
subject to
\begin{eqnarray}\label{xn39}
G_j\in\rt^{m\times r_j} \qa H_j\in\rt^{r_j\times q_j},   
\end{eqnarray}
and
\begin{eqnarray}\label{gg09}
\hspace*{10mm}E_{z_i z_j} =\oo,\quad \text{for $i\neq j$},
\end{eqnarray}
where  $i,j=0,\ldots,p$ and  $\mathbb{O}$ denotes the zero matrix (and the zero vector).

 It will be shown in Section \ref{n5o8} below that the solution of problem (\ref{gxd11}) - (\ref{gg09}) is determined under a special condition imposed on vectors  $\vv_1,\ldots,\vv_p$.


\subsection{Specific features of problem in (\ref{gxd11})-(\ref{gg09})}\label{nq8q09}



The proposed approximating operator $\ttt_p$ has several degrees of freedom to minimize the associated error. They are:

- `degree' $p$ of $\ttt_p$,

- matrices $G_0, H_0, \ldots,$ $G_p, H_p$,

-  injections $\vv_1,\ldots,\vv_p$,

 - values of $r_0,\ldots, r_p$ in (\ref{hhx2}), and

 -  dimensions $q_1,\ldots,q_p$ of  injections $\vv_1,\ldots, \vv_p$.

  It is shown in Sections \ref{yy41}, \ref{err1} and \ref{fk92} below that both the optimal choice of  $G_0, H_0, \ldots,$ $G_p, H_p$ and injections $\vv_1,\ldots, \vv_p$, and the increase in  $r_0,\ldots, r_p$ and  $q_1,\ldots,q_p$ leads to the decrease in the error associated with approximating operator $\ttt_p$.

Injections $\vv_1,\ldots,\vv_p$ represent a new special feature of the proposed technique.

Further, we would like to mention that the utility of transformations $Q_0,$ $\ldots, Q_{p}$ is twofold. First, they lead to a faster  numerical realization of the proposed method.  This is because the transformations imply the condition (\ref{gg09}) which allows to represent  problem in (\ref{gxd11}) and (\ref{xn39}) as a set of simpler problems each of which depends on the single pair $G_j, H_j$, for $j=0,\ldots,p$ (see (\ref{vbn34}) and (\ref{sad334}) below). It allows us to
avoid numerical difficulties associated with computation of  large  matrices.
Details are given in Section \ref{svd} that follows. Second,  transformations $Q_0,$ $\ldots, Q_{p}$ allow us to  determine injections $\vv_1,\ldots, \vv_p$  in the optimal way developed in Section \ref{jkw8} below.

\section{Preliminary results}\label{n5o8}

Here, we consider the determination of  pairwise uncorrelated vectors $\zz_1,\ldots,$ $\zz_p$ and the solution of a particular case of the problem in (\ref{gxd11}),  (\ref{xn39}), (\ref{gg09}) where a minimization with respect to $\zz_1,\ldots,$ $\zz_p$ is not included.  These preliminary results will be used in in Section \ref{hh58} where the  solution of the original problem represented by (\ref{gxd11}),  (\ref{xn39}), (\ref{gg09}) is provided.

\begin{definition}\label{93mm}
Random vectors $\zz_0,\ldots,\zz_p$ are called {\em pairwise uncorrelated} if  the condition in (\ref{gg09})  holds for any pair of vectors $\zz_i$ and $\zz_j$, for $i\neq j$, where $i,j=0,\ldots,p$. Two vectors $\zz_i$ and $\zz_j$ belonging to the set of the pairwise uncorrelated vectors are called {\em uncorrelated.}
\end{definition}

For $j=0,\ldots,p$, let ${\mathcal N} ({ M}^{(j)})$ be a null space of matrix ${ M}^{(j)}\in \rt^{q_j\times q_j}$.

\begin{definition}\label{90cm}
 Random vectors $\vv_0,\ldots, \vv_p$ are called {\em jointly independent} if
$$
{M}^{(0)}\vv_0(\omega) + \ldots + {M}^{(p)}\vv_p(\omega) =\oo,
$$
almost everywhere in $\Omega$, only if $\vv_j(\omega)\in {\mathcal N} ({ M}^{(j)})$, for $j=0,\ldots,p$.
\end{definition}

\begin{definition}\label{io29}
Random vector $\vv_j$, for $j=0,\ldots,p,$ is called the {\em well-defined injection} if
\begin{eqnarray}\label{mku}
\it\Gamma_{z_j} =  E_{xz_j}E_{z_jz_j}^\dagger E_{z_jx} \neq \oo,
\end{eqnarray}
where ${\bf z_j}$ is defined by (\ref{772x2}). Otherwise, injection $\vv_j$ is called ill-defined.
\end{definition}

 An explanation for introducing  Definition \ref{io29} is provided by Remark \ref{opw8} below.


\subsection{Determination of pairwise uncorrelated vectors }\label{kla7}

\begin{theorem}\label{902n}
Let random vectors $\vv_0,\ldots, \vv_p$ be jointly independent. Then they are transformed to the pairwise uncorrelated vectors $\zz_0,\ldots,\zz_p$ by transformations $Q_0,\ldots,Q_p$ as follows:
\begin{eqnarray}\label{nma0}
&&\zz_0= Q_0(\vv_0) =\vv_0 \quad \text{and, for $j=1,\ldots,p$,}\\
&&\zz_j = Q_j(\vv_j, Z_{j-1}) = \vv_j -\sum_{k=0}^{j-1} E_{v_j z_k}E_{z_k z_k}^\dag \zz_k.\label{nm92}
\end{eqnarray}
\end{theorem}

 \begin{proof}
Suppose  that the condition in (\ref{gg09}) holds for $\zz_0,$ $\ldots,\zz_{i-1}$. Then, for $\ell = 0,\ldots, i-1$,
\begin{multline}\label{eo2n}
E_{z_i z_\ell}=E[(\vv_i - \sum_{l=0}^{i-1} E_{v_i z_l}E_{z_l z_l}^\dag \zz_l)\zz_\ell^T]\\
=E_{v_i z_\ell} - \sum_{l=0}^{i-1} E_{v_i z_l}E_{z_l z_l}^\dag  E_{z_l z_\ell}\\
=E_{v_i z_\ell} -  E_{v_i z_\ell}E_{z_\ell z_\ell}^\dag  E_{z_\ell z_\ell}=\oo.
\end{multline}
The latter is true because  by Lemma 1 in \cite{2001390},
$$
E_{v_i z_\ell}E_{z_\ell z_\ell}^\dag  E_{z_\ell z_\ell} = E_{v_i z_\ell}.
$$
 Thus, by  induction, (\ref{gg09})  holds for any $i=0,\ldots,p$.  $\hfill\blacksquare$
 \end{proof}

The solution device of the problem in (\ref{gxd11})-(\ref{gg09}) is based, in particular, on the solution of the problem
\begin{eqnarray}\label{44k1}
\min_{\substack{G_0, H_0,\ldots, G_p, H_p\\}}  \left\|\f(\y) -  \sum_{j=0}^{p}G_j H_j \zz_j\right\|^2_\Omega
\end{eqnarray}
subject to (\ref{xn39}) and (\ref{gg09}). In the following Section \ref{svd}, matrices  $G_0, H_0, \ldots,$ $G_p, H_p$ that solve this problem  are given.


\subsection{Determination of matrices  $G_0, H_0, \ldots,$ $G_p, H_p$ that solve the problem in (\ref{44k1}), (\ref{xn39}) and (\ref{gg09})}\label{svd}

First, recall the definition of a truncated SVD.
Let the SVD of matrix  $A\in \rt^{m\times n}$ be given by
$
A=U_A\Sigma_A V_A^T,
 $
 where  $U_A=[u_1 \;u_2\;\ldots u_m]\in \rt^{m\times m}, V_A=[v_1 \;v_2\;\ldots v_n]\in \rt^{n\times n}$  are unitary matrices, and  $\Sigma_A=\diag(\sigma_1(A),$ $\ldots,$ $\sigma_{\min(m,n)}(A))\in\rt^{m\times n}$
 is a generalized diagonal matrix,  with the singular values $\sigma_1(A)\ge \sigma_2(A)\ge\ldots\ge 0$ on the main diagonal.
 For $k<m$, $j<n$ and $\ell<\min(m,n)$, we denote
 $$
 U_{A,k}=[u_1 \;u_2\;\ldots u_k],\quad V_{A,j}=[v_1 \;v_2\;\ldots v_j], \quad \Sigma_{A,\ell}=\diag(\sigma_1(A), \ldots,\sigma_\ell(A)),
 $$
  and write
  $$
  \displaystyle {\it\Pi}_{A,L}=\hspace*{-4mm}\sum_{k=1}^{\rank (A)}u_k u_k^T \qa \displaystyle  {\it\Pi}_{A,R}=\hspace*{-4mm} \sum_{j=1}^{\rank (A)}v_j v_j^T.
  $$
For $r=1,\ldots,\rank (A)$,
 \begin{eqnarray}\label{cdd03}
 \displaystyle [A]_r= \sum_{i=1}^{r}\sigma_i(A)u_i v_i^T\in \rt^{m\times n},
 \end{eqnarray}
 is the truncated SVD of $A$. For $r\geq\rank (A)$ we write $[A]_r=A\;(=A_{\rank (A)})$.

\begin{theorem}\label{9772n}
 Let  $\vv_0,\ldots,\vv_p$ be  well-defined injections and vectors ${\zz}_0,\ldots,{\zz}_p$ be pairwise uncorrelated.
Then the minimal Frobenius norm solution to the problem in (\ref{44k1})  is given, for $j=0,\ldots,p$, by
\begin{eqnarray}\label{73n32}
G_j = U_{{\it\Gamma}_{z_j},r_j} \qa H_j=U_{{\it\Gamma}_{z_j},r_j}^T  E_{{x}{z_j}}E_{{z_j}{z_j}}^\dag,
\end{eqnarray}

\end{theorem}

\begin{proof}
For $j=0,\ldots,p$, let $S_j = G_j H_j$, and let ${S = [S_0,\ldots, S_p]}$ and $\w = [\zz_0^T,\ldots,$ $\zz_p^T]^T$.
Then, for $\x=\f(\y)$,
\begin{eqnarray}\label{123sd}
\|\f(\y) - \sum_{j=0}^p S_j\zz_j\|^2_\Omega = \mbox{\em tr}\left\{(E_{xx} - E_{xw}S^T - SE_{wx} + SE_{ww}S^T)\right\},
\end{eqnarray}
where $E_{xw}=[E_{xz_0},\ldots,  E_{xz_p}]$ and by Theorem \ref{902n}, matrix $E_{ww}$ is block-diagonal,
$E_{ww}=\text{\em diag} [E_{z_0 z_0}, \ldots, E_{z_p z_p}].$
Thus,
$$
SE_{ww}S^T=S_0E_{z_0z_0}S_0^T+\ldots + S_pE_{z_pz_p}S_p^T,$$
$$ SE_{wx}=S_0E_{z_0x}+\ldots +S_pE_{z_px}.$$
Therefore,  (\ref{123sd}) implies 
\begin{eqnarray}\label{vbn34}
\|\f(\y) - \sum_{j=0}^p S_j\zz_j\|^2_\Omega  = \sum_{j=0}^p \|{\f(\y)}-S_j{\bf z_j}\|^2_\Omega -  \mbox{\em tr}\left\{p \hspace*{0.3mm} E_{xx}\right\},
\end{eqnarray}
where
\begin{eqnarray}\label{erc34}
&&\hspace*{-10mm} \|{\F(\y)}-S_j{\bf z_j}\|^2_\Omega =  \|{\F(\y)}-S_j{\bf z_j}\|^2_\Omega \nonumber\\
 &&  \hspace*{0mm} = \mbox{\em tr} \{E_{xx} - E_{xz_j}S_j^T - S_jE_{z_jx} +  S_j E_{z_jz_j} S_j^T\}\nonumber \\
&& \hspace*{0mm} = \|E_{xx}^{1/2}\|^2 -\|E_{xz_j}(E_{z_jz_j}^{1/2})^\dag\|^2+ \|(S_j - E_{xz_j}E_{z_jz_j}^{\dag}) E_{z_jz_j}^{1/2}\|^2 \nonumber \\
&& \hspace*{0mm} = \|E_{xx}^{1/2}\|^2 - \|E_{xz_j}(E_{z_jz_j}^{1/2})^\dag\|^2 +\|S_jE_{z_jz_j}^{1/2} - E_{xz_j}(E_{z_jz_j}^{1/2})^{\dag}\|^2
\end{eqnarray}
because
$$
E_{z_jz_j}^{\dag} E_{z_jz_j}^{1/2} = (E_{z_jz_j}^{1/2})^{\dag}
$$
and
\begin{equation}\label{62-ee}
E_{xz_j}E_{z_jz_j}^\dag E_{z_jz_j} = E_{xz_j}
\end{equation}
(see \cite{torbook2007}).

Let us denote by $\mathbb{R}_{r_j}^{m\times n}$ the set of all $m\times n$ matrices of rank at most $r_j$. In the RHS of (\ref{erc34}), only the last term depends on $S_j$. Therefore,  on the basis of \cite{tor5277,tor843,Torokhti2007,Liu201777,140977898,Wang201598}, the minimal Frobenius  norm solution to the problem
\begin{equation}\label{sad334}
\min_{S_j\in\mathbb{R}_{r_j}^{m\times n}}\|{\f(\y)}- S_j{\bf z}_j\|^2_\Omega,
\end{equation}
for $j=0,\ldots,p$, is given by
\begin{equation}\label{zs1}
S_j = G_j H_j= U_{{\it\Gamma}_{z_j},r_j}U_{{\it\Gamma}_{z_j},r_j}^T  E_{{x}{z_j}}E_{{z_j}{z_j}}^\dag.
\end{equation}
Then (\ref{73n32}) follows from (\ref{zs1}). $\hfill\blacksquare$
\end{proof}

\begin{remark}\label{opw8}
Definition \ref{io29} of the well-defined injections is motivated by the following observation.
It follows from  (\ref{73n32}) that  if, for all $j=0,\ldots,p,$  vector $\vv_j$ is such that ${\it\Gamma}_{z_j}=\oo$, then $G_j=\oo$ and $H_j=\oo$. In other words, then approximating operator $\ttt_p=\oo$.

Therefore,  in Theorem \ref{9772n} above and in the theorems below, vectors $\vv_0,\ldots, \vv_p$ are assumed well-defined.
\end{remark}



\subsection{Error analysis associated with the solution of  problem in (\ref{44k1}), (\ref{xn39}) and (\ref{gg09})}\label{err1}

In Theorem \ref{nm209} of this section, we obtain the constructive representation of the error associated with the solution of  problem in (\ref{44k1}), (\ref{xn39}) and (\ref{gg09}). In Theorem \ref{nm339}, we  show that the error can be improved by the increase in the dimensions of injections $\vv_1,\ldots,\vv_p$. Further,  Theorems \ref{nm209}, \ref{nm339} and \ref{nm779} establish  that the error is also diminished by the increase in the degree of approximating operator $\ttt_p$.

The error associated with  $\ttt_p(\vv_0,\ldots, \vv_p) = \displaystyle
   \sum_{j=0}^p G_j H_j\zz_j$ is denoted by
$$
\varepsilon^{(p)}_{_{GH}}=  \displaystyle \min_{G_0, H_0,\ldots, G_p, H_p}  \|\F(\y) - \sum_{j=0}^p G_j H_j\zz_j\|^2_\Omega.
$$

Let us denote the Frobenius norm by $\|\cdot\|$.

\begin{theorem}\label{nm209}
For $j=0,\ldots,p$, let $A_j = E_{{x} {z}_j} (E_{{z}_j {z}_j}^{1/2})^{\dag}$,   $\rank (A_j) = s_j$ and $s_j\geq r_j+1$. For $k=1,\ldots, s_j$, let $\sigma_k(A_j)$ be a singular value of $A_j$. Let $G_0, H_0,\ldots, G_p, H_p$ be determined by (\ref{73n32}). Then
\begin{eqnarray}\label{nb5n}
\varepsilon^{(p)}_{_{GH}}
=\mbox{\em tr} \{E_{{x} {x}}\} - \sum_{j=0}^{p}\sum_{k=1 }^{r_j}\sigma_k^2 (A_j)
\end{eqnarray}
In particular, the error  decreases as $p$ increases.
\end{theorem}

\begin{proof}
In the notation introduced in (\ref{cdd03}), matrix $G_j H_j$ in (\ref{73n32}) is represented as $G_j H_j = [E_{xz_j} (E_{z_jz_j}^{1/2})^{\dag}]_{r_j} (E_{z_jz_j}^{1/2})^\dag$. Therefore in (\ref{erc34}),
\begin{eqnarray}\label{62kk}
&&\hspace*{-10mm}\|G_j H_jE_{z_jz_j}^{1/2} - E_{xz_j}(E_{z_jz_j}^{1/2})^{\dag}\|^2\nonumber\\
&&\hspace*{30mm}= \|[E_{xz_j} (E_{z_jz_j}^{1/2})^{\dag}]_{r_j} (E_{z_jz_j}^{1/2})^\dag E_{z_jz_j}^{1/2} - E_{xz_j}(E_{z_jz_j}^{1/2})^{\dag}\|^2\nonumber\\
&&\hspace*{30mm}= \|[E_{xz_j} (E_{z_jz_j}^{1/2})^{\dag}]_{r_j}  - E_{xz_j}(E_{z_jz_j}^{1/2})^{\dag}\|^2\nonumber\\
&&\hspace*{30mm}= \|[A_j]_{r_j}  - A_j\|^2\nonumber\\
&& \hspace*{30mm}= \sum_{k=r_j +1 }^{s_j}\sigma_k^2 (A_j).
\end{eqnarray}
because $[E_{xz_j} (E_{z_jz_j}^{1/2})^{\dag}]_{r_j} (E_{z_jz_j}^{1/2})^\dag E_{z_jz_j}^{1/2}  = [E_{xz_j} (E_{z_jz_j}^{1/2})^{\dag}]_{r_j}$.
Further, since
\begin{eqnarray}\label{62ves}
\|E_{xz_j}(E_{z_jz_j}^{1/2})^\dag\|^2 = \|A_j\|^2 = \sum_{k=1 }^{s_j} \sigma_k^2 (A_j),
\end{eqnarray}
then (\ref{vbn34}), (\ref{erc34}), (\ref{62kk}) and (\ref{62ves}) imply (\ref{nb5n}).
$\hfill\blacksquare$
\end{proof}

Let us write
$$
A_j = \{a_{k i (j)}\}_{k,i=1}^{m, q_j}\qa A_j - [A_j]_{r_j}= \{b_{k i (j)}\}_{k,i=1}^{m, q_j}
$$
where $a_{k i (j)}$ and $a_{k i (j)}$ are entries of matrices $A_j$  and $A_j - [A_j]_{r_j}$, respectively.
Let us also denote
$$
\displaystyle\gamma_{k,(j)} = \sum_{i=1}^m \left(a_{k i (j)}^2 - b_{k i (j)}^2\right), \quad \displaystyle\gamma_{(j)} = \max \{\gamma_{1, (j)},\ldots, \gamma_{q_j, (j)}\}, \quad \displaystyle\gamma = \max_{j=1,\ldots,p}\gamma_{(j)},
$$
$
\displaystyle\alpha_{_0} = \mbox{ tr} \{E_{{x} {x}} - \sum_{j=0}^{p} A_j A_j^T \}$
 and $ \quad q=q_1 +\ldots, +q_p$.

In the following theorem we show that injections  $\vv_1,\ldots,\vv_p$ are useful  in the sense as their dimensions increase, so the error associated with the solution of  problem in (\ref{44k1}), (\ref{xn39}) and (\ref{gg09}) is diminished.

\begin{theorem}\label{nm339}
Let $\vv_1,\ldots,\vv_p$ be well-defined injections and let matrices $G_0,H_0,$ $\ldots,$ $G_p, H_p$ be defined by Theorem \ref{9772n}. Then the associated error decreases as the sum $q$ of dimensions  of injections $\vv_1,\ldots,\vv_p$ increases. In particular, there is $\beta \in (0, \gamma]$ such that, given $\alpha\geq \alpha_{_0}$, then
\begin{eqnarray}\label{rrr22}
\alpha_{_0}\leq \varepsilon^{(p)}_{_{GH}} \leq \alpha
\end{eqnarray}
iff
\begin{eqnarray}\label{xn22}
 q\geq \frac{\mbox{\em tr}\left\{E_{xx}\right\} - \displaystyle \sum_{k=1 }^{r_0}\sigma_k^2 (A_0)}{\beta}.
\end{eqnarray}

\end{theorem}

\begin{proof}
It follows from (\ref{vbn34}), (\ref{nb5n}), (\ref{erc34}) and (\ref{62kk}) that
\begin{eqnarray}\label{e55e}
\varepsilon^{(p)}_{_{GH}}&=& \|{\F(\y)}-G_j H_0{\bf z_0}\|^2_\Omega + \sum_{j=1}^p \|{\F(\y)}-G_j H_j{\bf z_j}\|^2_\Omega -  \mbox{\em tr}\left\{p \hspace*{0.3mm} E_{xx}\right\}\nonumber\\
&=& \mbox{\em tr}\left\{ E_{xx}\right\} - \sum_{k=1 }^{r_0}\sigma_k^2 (A_0) + \sum_{j=1}^p \|{\F(\y)}-G_j H_j{\bf z_j}\|^2_\Omega -  \mbox{\em tr}\left\{p \hspace*{0.3mm} E_{xx}\right\}
\end{eqnarray}
where
\begin{eqnarray*}
\|{\F(\y)}-G_j H_j{\bf z_j}\|^2_\Omega &=& \mbox{\em tr}\left\{ E_{xx}\right\} - \left(\left\| A_j\right\|^2 -  \left\|[A_j]_{r_j}  - A_j\right\|^2\right)\nonumber\\
 &=& \mbox{\em tr}\left\{ E_{xx}\right\} - \sum_{k=1}^{q_j}\sum_{i=1}^m \left(a_{k i (j)}^2 - b_{k i (j)}^2\right).
\end{eqnarray*}
Therefore,
\begin{eqnarray}\label{io22}
\varepsilon^{(p)}_{_{GH}}&=&\mbox{\em tr}\left\{ E_{xx}\right\} - \sum_{k=1 }^{r_0}\sigma_k^2 (A_0) - \sum_{j=1}^p \sum_{k=1}^{q_j} \sum_{i=1}^m \left(a_{k i (j)}^2 - b_{k i (j)}^2\right).
\end{eqnarray}
Here, $\displaystyle \sum_{k=1}^{q_j} \sum_{i=1}^m \left(a_{k i (j)}^2 - b_{k i (j)}^2\right) > 0$ since by (\ref{62kk}),
 \begin{eqnarray}\label{ebt5e}
 \left\| A_j\right\|^2 -  \left\|[A_j]_{r_j}  - A_j\right\|^2 = \sum_{k=1 }^{r_j}\sigma_k^2 (A_j) >0.
\end{eqnarray}
Thus, (\ref{e55e}) - (\ref{ebt5e}) imply that  $\varepsilon^{(p)}_{_{GH}}$ decreases as $q_j$ increases, for $j=1,\ldots, p,$ and $p$ increases.

Further, (\ref{nb5n}) implies
 \begin{eqnarray}\label{io21}
\varepsilon^{(p)}_{_{GH}}\geq \mbox{\em tr} \{E_{{x} {x}}\} - \sum_{j=0}^{p}\sum_{k=1 }^{s_j}\sigma_k^2 (A_j)
 = \mbox{\em tr} \left\{E_{{x} {x}} - \sum_{j=0}^{p} A_j A_j^T \right\} = \alpha_{_0}.
\end{eqnarray}
 Since
 $$
 0 < \sum_{j=1}^p \sum_{k=1}^{q_j} \sum_{i=1}^m \left(a_{k i (j)}^2 - b_{k i (j)}^2\right) \leq \gamma\sum_{j=1}^p q_j = \gamma q,
 $$
 then
 \begin{eqnarray}\label{io23}
\sum_{j=1}^p \sum_{k=1}^{q_j} \sum_{i=1}^m \left(a_{k i (j)}^2 - b_{k i (j)}^2\right) = q \beta.
 \end{eqnarray}
Therefore, (\ref{io22}), (\ref{io21}) and (\ref{io23}) imply
$$
 \alpha_{_0}\leq \varepsilon^{(p)}_{_{GH}} = \mbox{\em tr}\left\{ E_{xx}\right\} - \sum_{k=1 }^{r_0}\sigma_k^2 (A_0) - q \beta.
$$
thus, if $\varepsilon^{(p)}_{_{GH}} \leq \alpha$ then  (\ref{xn22}) is true. Conversely, if the latter is true then $\varepsilon^{(p)}_{_{GH}} \leq \alpha$.
$\hfill\blacksquare$
\end{proof}

\begin{remark}\label{hjk29}
An empirical explanation of Theorem \ref{nm339} is that the increase in $q$ implies the increase in the dimensions of matrices $H_1,\ldots, H_p$  in (\ref{hq1q2x2}) and (\ref{73n32}). Hence, it implies the increase in the number of parameters to optimize. As a result, for a fixed parameter $r$ given by (\ref{hhx2}), the  accuracy associated with approximating operator $\ttt_p$ improves. Further, it follows from (\ref{xn22}) that, as $q$ increases, $\varepsilon^{(p)}_{_{GH}} $ tends to $\alpha_0$ which is the error associated with the full rank approximating operator $\sss_h$ (see (\ref{hqzz2}) and (\ref{r4cn}) below).
\end{remark}

\begin{remark}\label{281c}
By Theorem \ref{nm209}, the error associated with solution of  problem (\ref{44k1}) decreases as degree $p$ of the approximating operator $\ttt_p$ increases.
At the same time, the increase in degree $p$ of approximating operator $\ttt_p$ may involve an increase in parameter $r$ (see (\ref{hhx2})). However, by a condition of some applied problem in hand, $r$ must be fixed. In the following Theorem \ref{nm779}, under the condition of  fixed $r$, the case of decreasing the error  as the degree $p$ of the approximating operator increases  is detailed.
\end{remark}

\begin{theorem}\label{nm779}
Let $r$ and $r_j$, for $j=0,\ldots,p$, be given. Let $g$ be a nonnegative integer such that $g<p$ and let $\ell_g = r_g + r_{g+1} +\ldots + r_p$. If
\begin{eqnarray}\label{cvb4}
\sum_{k=r_g+1 }^{r_{g+1} +\ldots + r_p}\sigma_k^2 (A_g) < \sum_{j=g+1}^{p}\sum_{k=1 }^{r_j}\sigma_k^2 (A_j),
\end{eqnarray}
where $\displaystyle \sum_{j=g+1}^{p}\sum_{k=1 }^{r_j}\sigma_k^2 (A_j) = \sum_{j=g+1}^{p} \sum_{k=1}^{q_j}\sum_{i=1}^m \left(a_{k i (j)}^2 - b_{k i (j)}^2\right)$,  then
\begin{eqnarray}\label{rr41}
\varepsilon^{(p)}_{_{GH}} < \varepsilon^{(g)}_{_{GH}},
\end{eqnarray}
i.e., for the same $r$,  the error associated with the approximating operator of  higher degree $p$ is less than the error associated with the approximating operator of lower degree $g$.
\end{theorem}

\begin{proof}
We write
$ 
r= r_0 + \ldots + r_{g-1} + \ell_g.
$ 
Then
\begin{eqnarray}\label{nb5z}
\varepsilon^{(g)}_{_{GH}}&=&\mbox{\em tr} \{E_{{x} {x}}\} - \sum_{j=0}^{g-1}\sum_{k=1 }^{r_j}\sigma_k^2 (A_j) - \sum_{k=1 }^{\ell_g}\sigma_k^2 (A_g)\nonumber\\
&=&\mbox{\em tr} \{E_{{x} {x}}\} - \sum_{j=0}^{g-1}\sum_{k=1 }^{r_j}\sigma_k^2 (A_j) - \sum_{k=1 }^{r_g}\sigma_k^2 (A_g) - \sum_{k=r_g+1 }^{r_{g+1} +\ldots + r_p}\sigma_k^2 (A_g).
\end{eqnarray}
Thus, (\ref{nb5n}) and (\ref{nb5z}) imply (\ref{cvb4}) and (\ref{rr41}).
$\hfill\blacksquare$

\end{proof}

\begin{remark}\label{90snm}
The RHS in (\ref{cvb4}) increases as the dimension $q_j$ of at least single  injection $\vv_j$,  for $j=g+1,\ldots,p$, increases while the LHS does not depend on $q_j$. In other words, one can always find $q_j$,  for $j=g+1,\ldots,p$, such that the inequality in
(\ref{cvb4}) is true.
\end{remark}

\subsection{Particular case: no  reduction of vector dimensionality}\label{yy41}

An  important particular case of the problem in (\ref{44k1}), (\ref{xn39}) and (\ref{gg09}) is when matrix $G_jH_j$, for $j=0,\ldots,p$, is replaced with a full rank matrix $P_j\in\rt^{m\times q_j}$. Then operator $\sss_h$ given by
\begin{eqnarray}\label{hqzz2}
 \sss_h(\vv_0,\ldots\vv_p) = \sum_{k=0}^{h}P_k\zz_k
\end{eqnarray}
is called {\em the full rank} approximating operator.

\begin{theorem}\label{nn32n}
Let vectors $\zz_0,\ldots,\zz_p$ be pairwise uncorrelated.   Then the minimal Frobenius norm solution to the problem
\begin{eqnarray}\label{gx7j}
\min_{P_0,\ldots,P_h} \|\F(\y) - \sum_{k=0}^{h}P_k\zz_k \|^2_\Omega,
\end{eqnarray}
 is given, for $k=0,\ldots,h$, by
\begin{eqnarray}\label{wun32}
P_k =  E_{x{z_k}}E_{{z_k}{z_k}}^\dag.
\end{eqnarray}

\end{theorem}

\begin{proof}
Similar to (\ref{vbn34}),
\begin{eqnarray}\label{vb744}
\|\F(\y) - \sum_{j=0}^h P_j\zz_j\|^2_\Omega  = \sum_{j=0}^h\|{\bf x}-P_j{\bf z_j}\|^2_\Omega -  \mbox{\em tr}\left\{h \hspace*{0.3mm} E_{xx}\right\}.
\end{eqnarray}
It is known (see, for example, \cite{905856,Brillinger2001,Jolliffe2002}) that the  minimal Frobenius  norm solution to the problem
$$
\min_{P_j}\|{\bf x}-P_j{\bf z_j}\|^2_\Omega
$$
is given by (\ref{wun32}). $\hfill\blacksquare$
\end{proof}


\begin{theorem}\label{bn2n}
Let $A_j = E_{xz_j} (E_{z_jz_j}^{1/2})^{\dag}$.
The error associated with  the minimal Frobenius norm solution to the problem in (\ref{gx7j}) is represented by
\begin{eqnarray}\label{r4cn}
\min_{P_0,\ldots,P_h} \|\F(\y) - \sum_{j=0}^{h}P_j\zz_j \|^2_\Omega =\mbox{\em tr} \{E_{xx}\}  - \sum_{j=0}^{h}\|A_j\|^2.
\end{eqnarray}
\end{theorem}

\begin{proof}
For $P_j$  determined by (\ref{wun32}),
\begin{eqnarray}\label{ioxb3}
 &&\hspace*{-15mm}\|\F(\y) - P_j\zz_j \|^2_\Omega \nonumber\\
&& \hspace*{05mm} = \mbox{\em tr} \{E_{xx}\}  - \|E_{xz_j}(E_{z_jz_j}^{1/2})^\dag\|^2 +\|P_jE_{z_jz_j}^{1/2} - E_{xz_j}(E_{z_jz_j}^{1/2})^{\dag}\|^2.
\end{eqnarray}
Here,
\begin{eqnarray}\label{en2b3}
&&\hspace*{-25mm}\|P_jE_{z_jz_j}^{1/2} - E_{xz_j}(E_{z_jz_j}^{1/2})^{\dag}\|^2\nonumber\\
&& \hspace*{0mm}= \|E_{xz_j} E_{z_jz_j}^{\dag}E_{z_jz_j}^{1/2} - E_{xz_j}(E_{z_jz_j}^{1/2})^{\dag}\|^2\nonumber\\
&&\hspace*{20mm} = \|E_{xz_j} (E_{z_jz_j}^{1/2})^{\dag} - E_{xz_j}(E_{z_jz_j}^{1/2})^{\dag}\|^2 = 0.\nonumber
\end{eqnarray}
That is,
\begin{eqnarray}\label{88sb3}
\|\F(\y) - P_j\zz_j \|^2_\Omega = \mbox{\em tr} \{E_{xx}\}  - \|E_{xz_j}(E_{z_jz_j}^{1/2})^\dag\|^2.
\end{eqnarray}
Then (\ref{r4cn}) follows from (\ref{vbn34}) and (\ref{en2b3}). $\hfill\blacksquare$

\end{proof}

\section{Solution  of problem given by (\ref{gxd11}),  (\ref{xn39}), (\ref{gg09})}\label{hh58}

Now we are in the position to consider a solution of the original problem in (\ref{gxd11}),  (\ref{xn39}), (\ref{gg09}).

\subsection{Device of solution}\label{nm33}


In comparison with the problem in (\ref{44k1}),  (\ref{xn39}), (\ref{gg09}), a specific difficulty of the original problem in (\ref{gxd11}), (\ref{xn39}), (\ref{gg09}) is a determination of $p$ additional unknowns, injections $\vv_1,\ldots,\vv_{p}$.

The device of the proposed solution is as follows.
 First,  in (\ref{nma0})-(\ref{nm92}), the pairwise uncorrelated random vectors $\zz_0,\ldots,\zz_p$ are denoted by $\zz_0^{(0)}\ldots,\zz^{(0)}_p$. In (\ref{73n32}), matrices $G_j$ and $H_j$,  for $j=0,\ldots,p$,  are denoted  by  $G^{(0)}_j$ and $H^{(0)}_j$, respectively. The associated error is still represented by (\ref{nb5n}).

 Then, for $i=0,1,\ldots$, searched injections $\vv^{(i+1)}_1,$ $\ldots, \vv^{(i+1)}_p$ and  matrices $G^{(i+1)}_1,$ $H^{(i+1)}_1$, $\ldots G^{(i+1)}_p, H^{(i+1)}_p$  are determined by the iterative procedure represented bellow. In particular, it will be shown in Theorem \ref{opwm9}  below  that $\vv^{(i+1)}_1,$ $\ldots, \vv^{(i+1)}_p$ and  $G^{(i+1)}_1,$ $H^{(i+1)}_1$, $\ldots G^{(i+1)}_p, H^{(i+1)}_p$ further minimize the associated error. The $i$-th  iterative loop of the iterative procedure consists of the steps  as follows.

\medskip
{\em The $i$-th  iterative loop, for $i=0,1,\ldots.$}
\medskip

{\em Step 1.} If $i=0$, denote $ \widetilde{\zz}^{(0)}_k = {\zz}^{(0)}_k$, for $k=1,\ldots,p$, and   go to Step 3. Otherwise, go to Step 2.

{\em Step 2.} Given  $\vv^{(i)}_1,\ldots,\vv^{(i)}_{p}$,   find pairwise uncorrelated random vectors $\zz_1^{(i)}\ldots,\zz^{(i)}_p$. Recall that $\zz_0=\vv_0=\y$ and that $\zz_1^{(i)}\ldots,\zz^{(i)}_p$ are such that $E_{z^{(i)}_k z^{(i)}_j} =\oo,$ for $k\neq j$. The latter condition implies a simplification of the computational procedure for determining $G_1, H_1, \ldots, G_p, H_p$ in Step 5 below. Since $\zz_1^{(i)}\ldots,\zz^{(i)}_p$ are not determined from an error minimization problem, they are updated in the next Step 3.

{\em Step 3.}
Given  $\vv_0=\y$,    $G^{(i)}_0, H^{(i)}_0, \ldots,$ $G^{(i)}_p, H^{(i)}_p$, find $\zz_1,\ldots, \zz_p$ that solve
\begin{eqnarray}\label{gxa1j}
\min_{{\mathbf z}_1,\ldots, {\mathbf z}_p} \|\F(\y) - \sum_{k=1}^{p} G^{(i)}_k H^{(i)}_k \zz_k\|^2_\Omega.
\end{eqnarray}
 The solutions are denoted by $\widetilde{\zz}^{(i+1)}_1, \ldots, \widetilde{\zz}^{(i+1)}_p$.
Also denote
\begin{eqnarray}\label{zt7j}
\varepsilon^{(i+1)}_{z} = \|\F(\y) -  \sum_{k=0}^{p}G^{(i)}_k H^{(i)}_k \widetilde{\zz}^{(i+1)}_k\|^2_\Omega,
\end{eqnarray}
where, for $k=0$, we set $G^{(i)}_0 = G^{(0)}_0$ and $H^{(i)}_0 = H^{(0)}_0$, for all $i=1,2,\ldots $.

{\em Step 4.} For $\widetilde{Z}_{j-1} = \{\widetilde{\zz}^{(i+1)}_0, \ldots, \widetilde{\zz}^{(i+1)}_{j-1}\}$, find $\vv_j$  that solves the equation
 \begin{eqnarray}\label{12x2}
\widetilde{\zz}^{(i+1)}_j = Q_j(\vv_j, \widetilde{Z}_{j-1}),
\end{eqnarray}
 for $j=1,\ldots,p$,  and denote it by $\vv^{(i+1)}_j$. We call $\vv^{(i+1)}_j$ the {\em optimal injection}.

{\em Step 5.} Given $\widetilde{\zz}^{(i)}_1, \ldots, \widetilde{\zz}^{(i)}_p$, find $G_1, H_1, \ldots, G_p, H_p$ that solve
\begin{eqnarray}\label{vvh7j}
\min_{\substack{G_1, H_1,\ldots, G_p, H_p\\}}  \left\|\F(\y) -  \sum_{k=1}^{p}G_k H_k \widetilde{\zz}^{(i)}_k\right\|^2_\Omega.
\end{eqnarray}
 The solution of problem (\ref{vvh7j}) is denoted by   $G^{(i+1)}_1, H^{(i+1)}_1, \ldots,$ $G^{(i+1)}_p, H^{(i+1)}_p$.
Further, denote
 \begin{eqnarray}\label{vt7j}
\varepsilon^{(i+1)}_{GH} = \|\F(\y) -  \sum_{k=0}^{p}G^{(i+1)}_k H^{(i+1)}_k \widetilde{\zz}^{(i)}_k\|^2_\Omega,
\end{eqnarray}
where, as before, for $k=0$, we set $G^{(i+1)}_0 = G^{(0)}_0$ and $H^{(i+1)}_0 = H^{(0)}_0$, for all $i=0,1,\ldots $.

{\em Step 6.} If
$
\varepsilon^{(i+1)}_{z} = \min \{\varepsilon^{(i+1)}_{GH}, \varepsilon^{(i+1)}_{z}\}
$
 then denote
\begin{eqnarray}\label{t87j}
 \varepsilon^{(i+1)} = \varepsilon^{(i+1)}_{z}.
 \end{eqnarray}
 If
$
\varepsilon^{(i+1)}_{GH} = \min \{\varepsilon^{(i+1)}_{GH}, \varepsilon^{(i+1)}_{z}\}
$
 then set
  \begin{eqnarray}\label{bb7j}
 \varepsilon^{(i+1)} = \varepsilon^{(i+1)}_{GH}.
  \end{eqnarray}

{\em Step 7.}
If, for a given tolerance $\delta$,
\begin{eqnarray}\label{xms7j}
 \|\varepsilon^{(i+1)} - \varepsilon^{(i)}\|^2_\Omega\leq \delta,
\end{eqnarray}
the iterations are stopped. If not then  Steps 2-7 are repeated to form the next iterative loop.

For $i=0,1,\ldots,$ we denote
\begin{eqnarray}\label{90qb}
\ttt_p^{(i, i+1)} (\vv_0, \vv^{(i+1)}_1,\ldots,\vv^{(i+1)}_p) = \sum_{k=0}^{p}G^{(i)}_k H^{(i)}_k \widetilde{\zz}^{(i+1)}_k
\end{eqnarray}
and
\begin{eqnarray}\label{90qb1}
\ttt_p^{(i+1, i)} (\vv_0, \vv^{(i)}_1,\ldots,\vv^{(i)}_p) = \sum_{k=0}^{p}G^{(i+1)}_k H^{(i+1)}_k \widetilde{\zz}^{(i)}_k.
\end{eqnarray}
The above steps of the solution device are consummated as follows.

\subsection{Determination of $\widetilde{\zz}^{(i+1)}_j$ and $\vv^{(i+1)}_j$ in Steps 3 and 4}\label{jkw8}

\begin {theorem}\label{2m7a}
For $i=0,1,\ldots$, $j=1,\ldots,p$ and all $\omega\in\Omega$, the solution of problem (\ref{gxa1j}) is given by
\begin{eqnarray}\label{22m9}
\widetilde{\zz}^{(i+1)}_j(\omega)=(G^{(i)}_j H^{(i)}_j)^\dagger \x(\omega).
\end{eqnarray}
\end{theorem}

\begin{proof}

Similar to (\ref{vbn34}), the cost function in (\ref{gxa1j}) is represented as follows:
\begin{eqnarray}\label{92a1j}
 \left\|{\F(\y)} - \sum_{j=0}^{p}G^{(i)}_j H^{(i)}_j \zz_j\right\|^2_\Omega
= \sum_{j=0}^{p}\left\|{\F(\y)}-G^{(i)}_j H^{(i)}_j{\bf z}_j\right\|^2_\Omega -\mbox{\em tr}\left\{ p E_{xx}\right\}.
\end{eqnarray}
Here,
\begin{eqnarray}\label{nm19}
\|{\F(\y)}-G^{(i)}_j H^{(i)}_j{\bf z}_j\|^2_\Omega
=\int_{\Omega} \|{\F(\y)}(\omega) -G^{(i)}_j H^{(i)}_j\zz_j(\omega) \|^2 d\mu(\omega).
\end{eqnarray}
For $j=1,\ldots,p$, vector $\zz_j(\omega)\in\rt^{q_j}$ of the smallest Euclidean norm of all minimizers that solves
$$
\min_{{\mathbf z}_j(\omega)}\|{\F(\y)}(\omega) -G^{(i)}_j H^{(i)}_j\zz_j(\omega) \|^2
$$
is given by (see \cite{golub1996}, p. 257)
\begin{eqnarray*}
\zz_j(\omega)=\widetilde{\zz}^{(i+1)}_j(\omega)=(G^{(i)}_j H^{(i)}_j)^\dag \x(\omega).
\end{eqnarray*}
Thus, (\ref{22m9}) is true. $\hfill\blacksquare$
\end{proof}

Further, we wish to find  optimal injections  $\vv^{(i+1)}_1,\ldots, \vv^{(i+1)}_p$, i.e., those that satisfy (\ref{12x2}) where $Q_j$ is represented by (\ref{nm92}) written in terms of  $\widetilde{\zz}^{(i+1)}_0,$ $\ldots,$ $\widetilde{\zz}^{(i+1)}_p$.
Then (\ref{nm92}) and (\ref{12x2}) imply, for $j=1,\ldots,p$,
\begin{eqnarray}\label{anmyy}
(G^{(i)}_j H^{(i)}_j)^\dag \x(\omega) = \vv_j(\omega) -\sum_{\ell=0}^{j-1} E_{v_j \widetilde{z}_\ell}E_{\widetilde{z}_\ell \widetilde{z}_\ell}^\dag  (G^{(i)}_\ell H^{(i)}_\ell)^\dag \x(\omega)\label{b419}.
\end{eqnarray}
 In the RHS of (\ref{anmyy}), for a notation simplicity, we write $\widetilde{z}_\ell$ instead of $\widetilde{z}^{(i+1)}_\ell$.

Thus, optimal injection  $\vv^{(i+1)}_j$, for $j=1,\ldots,p$, should satisfy  (\ref{anmyy}) where $\vv_j$  should be replaced with $\vv^{(i+1)}_j$. The latter is determined by Theorem \ref{553a} below where, for $i=0,1,\ldots,$ $j=1,\ldots,$ $p$ and  $\ell, k=0,\ldots,j$, we denote
\begin{eqnarray}\label{vmq8}
&&\hspace*{-13mm}\bgamma^{(i)}_\ell(\omega) = E_{\widetilde{z}^{(i+1)}_\ell \widetilde{z}^{(i+1)}_\ell}^\dag \widetilde{\zz}^{(i+1)}_\ell(\omega), \quad B^{(i)}_{j k} = (G^{(i)}_j H^{(i)}_j)^\dag    \sum_{s=0}^{h}F^{(0)}_{s}E_{\widetilde{z}^{(0)}_s \widetilde{z}^{(i+1)}_k}\\
&&\hspace*{17mm}\qa   A^{(i)}_{\ell k} = E_{\widetilde{z}^{(i+1)}_\ell \widetilde{z}^{(i+1)}_\ell}^\dag E_{\widetilde{z}^{(i+1)}_\ell \widetilde{z}^{(i+1)}_k}.
\end{eqnarray}
Here, $\bgamma^{(i)}_\ell(\omega)\in \rt^{q_\ell\times q_\ell}$, $B^{(i)}_{j k} \in\rt^{q_{j} \times q_{k}}$ and $ A^{(i)}_{\ell k}\in\rt^{q_\ell\times q_k}$.

\begin {theorem}\label{553a}
 Let $\widetilde{\zz}^{(i+1)}_0,\ldots,\widetilde{\zz}^{(i+1)}_p$  be given by (\ref{22m9}) and let
\begin{eqnarray}\label{rrs1}
 B^{(i)}=[B^{(i)}_{j 0},\ldots, B^{(i)}_{j,j-1}]  \qa
A^{(i)}= \{A^{(i)}_{k s}\}_{k,s=0}^{j-1},
\end{eqnarray}
where  $B^{(i)}\in\rt^{q_j\times q}$, $A^{(i)}\in\rt^{q\times q}$ and  $q=q_0 +\ldots +q_{j-1}$.
Then $\vv^{(i+1)}_j$, for $j=1,\ldots,p$,  is determined by
\begin{eqnarray}\label{3mc9s}
\vv^{(i+1)}_j(\omega)= (G^{(i)}_j H^{(i)}_j)^\dag \x(\omega) + \sum_{\ell=0}^{j-1} C^{(i)}_{j\ell} \bgamma^{(i)}_\ell(\omega),
\end{eqnarray}
where  matrices $C^{(i)}_{j 0},\ldots, C^{(i)}_{j,j-1}$ are defined by
\begin{eqnarray}\label{qq2s1}
[C^{(i)}_{j 0},\ldots, C^{(i)}_{j,j-1}] = B^{(i)}(I- A^{(i)})^{\dag}.
\end{eqnarray}
Here, $C^{(i)}_{j\ell} \in\rt^{q_j\times q_\ell}$, for $\ell = 0,\ldots, j-1$.
\end{theorem}

\begin{proof}
 We write (\ref{anmyy}) as
\begin{eqnarray*}
\vv_j(\omega)=(G^{(i)}_j H^{(i)}_j)^\dag \x(\omega) + \sum_{\ell=0}^{j-1} E_{v_j \widetilde{z}_\ell}E_{\widetilde{z}_\ell \widetilde{z}_\ell}^\dag  (G^{(i)}_\ell H^{(i)}_\ell)^\dag \x(\omega),
\end{eqnarray*}
i.e.,
\begin{eqnarray}\label{bq119}
\vv_j(\omega)= (G^{(i)}_j H^{(i)}_j)^\dag \x(\omega) + \sum_{\ell=0}^{j-1} \left[\int_{\Omega} \vv_j(\xi)[{\widetilde{\zz}^{(i+1)}_\ell}(\xi)]^T d\mu(\xi)\right] \bgamma^{(i)}_\ell(\omega).
\end{eqnarray}
Let us now replace $\vv_j(\omega)$  with $\vv^{(i+1)}_j(\omega)$ and write equation (\ref{bq119}) as follows:
\begin{eqnarray}\label{x5e19}
\vv^{(i+1)}_j(\omega)= \balpha^{(i)}_j(\omega) +  \int_{\Omega} K^{(i)}_j (\omega,\xi) \vv^{(i+1)}_j(\xi)d\mu(\xi),
\end{eqnarray}
where
\begin{eqnarray}\label{79e19}
\balpha^{(i)}_j(\omega) = (G^{(i)}_j H^{(i)}_j)^\dag \x(\omega) \qa \displaystyle K^{(i)}_j (\omega,\xi)=\sum_{\ell=0}^{j-1}[{\widetilde{\zz}^{(i+1)}_\ell}(\xi)]^T \bgamma^{(i)}_\ell(\omega).
\end{eqnarray}
Recall that in (\ref{79e19}), matrices $G^{(i)}_j$ and  $H^{(i)}_j$  depend on  $\vv^{(i)}_j(\omega)$, not on $\vv^{(i+1)}_j(\omega)$. Therefore, equation (\ref{x5e19}) is a vector version of the Fredholm integral equation of the second kind   \cite{kanwal1996} with respect to ${\vv}^{(i+1)}_j(\omega)$. Its solution is provided as follows. Write (\ref{x5e19}) as
\begin{eqnarray}\label{2mc9s}
\vv^{(i+1)}_j(\omega)= \balpha^{(i)}_{j}(\omega) + \sum_{\ell=0}^{j-1} C^{(i)}_{j\ell} \bgamma^{(i)}_\ell(\omega),
\end{eqnarray}
where
$$
\displaystyle C^{(i)}_{j\ell} = \int_{\Omega} \vv^{(i+1)}_j(\xi)[{\widetilde{\zz}^{(i+1)}_\ell}(\xi)]^T d\mu(\xi).
$$

Let us now multiply both sides of (\ref{2mc9s}) by $[\zz^{(i+1)}_k(\omega)]^T$, for $k=0,\ldots,j-1$,
 and integrate. It implies
\begin{eqnarray*}
& &\hspace*{-7mm} \int_{\Omega} \vv^{(i+1)}_j(\omega)[{\widetilde{\zz}^{(i+1)}_k}(\omega)]^Td\mu(\omega) \nonumber\\
& & = \int_{\Omega} \balpha^{(i)}_{j}(\omega)[{\widetilde{\zz}^{(i+1)}_k}(\omega)]^Td\mu(\omega) +  \sum_{\ell=0}^{j-1} C^{(i)}_{j\ell} \int_{\Omega}\bgamma^{(i)}_\ell(\omega)[{\widetilde{\zz}^{(i+1)}_k}(\omega)]^Td\mu(\omega)
\end{eqnarray*}
and
\begin{eqnarray}\label{2n44s}
C^{(i)}_{j k} = B_{j k}  + \sum_{\ell=0}^{j-1} C^{(i)}_{j\ell}A^{(i)}_{\ell k},
\end{eqnarray}
where, for $k=0,\ldots,j-1$,
\begin{eqnarray}\label{n444s}
 B^{(i)}_{j k}  &=& \int_{\Omega}  \balpha^{(i)}_{j}(\omega)[{\widetilde{\zz}^{(i+1)}_k}(\omega)]^T d\mu(\omega)\nonumber\\
&=& (G^{(i)}_j H^{(i)}_j)^\dag \int_{\Omega}  \left[\sum_{s=0}^{h}F^{(0)}_{s}{\zz}^{(0)}_{s}(\omega)\right][{\widetilde{\zz}^{(i+1)}_k}(\omega)]^T d\mu(\omega)\nonumber\\
&=&  (G^{(i)}_j H^{(i)}_j)^\dag    \sum_{s=0}^{h}F^{(0)}_{s}E_{{z}^{(0)}_s \widetilde{z}^{(i+1)}_k}  \in\rt^{q_{j} \times q_{k}}
\end{eqnarray}
 and
 \begin{eqnarray}\label{naa3s}
&&\hspace*{-10mm} A^{(i)}_{\ell k}=\int_{\Omega}\bgamma^{(i)}_\ell(\omega)[{\widetilde{\zz}^{(i+1)}_k}(\omega)]^Td\mu(\omega)\nonumber\\
 &&\hspace*{30mm} = E_{\widetilde{z}^{(i+1)}_\ell \widetilde{z}^{(i+1)}_\ell}^\dag E_{\widetilde{z}^{(i+1)}_\ell \widetilde{z}^{(i+1)}_k}\in\rt^{q_\ell\times q_k}.
\end{eqnarray}
Let  $C^{(i)}=[C^{(i)}_{j 0},\ldots, C^{(i)}_{j,j-1}]\in\rt^{q_j\times q}$. Then the set of matrix equations in (\ref{2n44s}) can be written as a single equation
\begin{eqnarray}\label{2cm4s}
C^{(i)}=B^{(i)} + C^{(i)} A^{(i)}
\end{eqnarray}
or
\begin{eqnarray}\label{2j44s}
B^{(i)} = C^{(i)}(I- A^{(i)}).
\end{eqnarray}
If matrix $I- A^{(i)}$ is invertible then (\ref{2j44s}) implies
\begin{eqnarray}\label{xn2s}
C^{(i)} = B^{(i)}(I- A^{(i)})^{-1}.
\end{eqnarray}
If matrix $I- A^{(i)}$ is singular then instead of equation (\ref{2cm4s}) or (\ref{2j44s}) we consider the  problem
\begin{eqnarray}\label{zb34s}
\min_{C^{(i)}}\|B^{(i)} - C^{(i)}(I- A^{(i)})\|^2.
\end{eqnarray}
Its minimal Frobenius norm solution is given by \cite{ben1974}
\begin{eqnarray}\label{xn2s1}
C^{(i)} = B^{(i)}(I- A^{(i)})^{\dag}.
\end{eqnarray}
Thus, injection $\vv^{(i+1)}_j$ we search follows from (\ref{2mc9s}),  (\ref{xn2s}) and (\ref{xn2s1}). $\hfill\blacksquare$
\end{proof}


\subsection{Determination of matrices  $G^{(i+1)}_1, H^{(i+1)}_1, \ldots,$ $G^{(i+1)}_p, H^{(i+1)}_p$ in Step 5}\label{nmwk1}

In Step 5, matrices $G^{(i+1)}_1, H^{(i+1)}_1, \ldots,$ $G^{(i+1)}_p, H^{(i+1)}_p$ solve problem (\ref{vvh7j}). At the same time,
the problems in  (\ref{44k1}) and (\ref{vvh7j})  differ in the notation only. Therefore, in Step 5, matrices $G^{(i+1)}_1, H^{(i+1)}_1, \ldots,$ $G^{(i+1)}_p, H^{(i+1)}_p$ are determined, in  fact, by Theorem \ref{9772n} where only the notation should be changed. Nevertheless, to avoid any confusion, we below provide Theorem \ref{xn12n} where matrices $G^{(i+1)}_1, H^{(i+1)}_1, \ldots,$ $G^{(i+1)}_p, H^{(i+1)}_p$ are represented.
To simplify the notation in Theorem \ref{xn12n} , we set
 \begin{eqnarray}\label{xnm33}
{\it\Gamma}_{{z}_j} = {\it\Gamma}_{{\widetilde z}^{(i)}_j},  \quad E_{{x}{{z}_j}} = E_{{x}{{\widetilde z}^{(i)}_j}} \qa E_{{{z}_j}{{ z}_j}} = E_{{{\widetilde z}^{(i)}_j}{{\widetilde z}^{(i)}_j}}.
 \end{eqnarray}

\begin{theorem}\label{xn12n}
 Let  $\vv_0,\vv^{(i)}_1,\ldots,\vv^{(i)}_p$ be  well-defined injections and vectors ${\widetilde\zz}^{(i)}_1,$ $\ldots,{\widetilde\zz}^{(i)}_p$ be pairwise uncorrelated.
Then in the notation (\ref{xnm33}), the minimal Frobenius norm solution to the problem in (\ref{vvh7j})  is given, for $j=1,\ldots,p$, by
\begin{eqnarray}\label{71n32}
G^{(i+1)}_j = U_{{\it\Gamma}_{{z}_j},r_j} \qa H^{(i+1)}_j=U_{{\it\Gamma}_{{ z}_j},r_j}^T  E_{{x}{{z}^{(i)}_j}}E_{{{ z}_j}{{z}_j}}^\dag.
\end{eqnarray}

\end{theorem}

\begin{proof}
The proof follows from the proof of Theorem \ref{9772n}. $\hfill\blacksquare$
\end{proof}

\subsection{Error analysis of the solution of problem in (\ref{gxd11}),  (\ref{xn39}), (\ref{gg09})}\label{fk92}

\begin{theorem}\label{opwm9}
Let $G_0^{(0)},$ $H_0^{(0)},$ and  $G_k^{(i)},$ $H_k^{(i)}$ and $\widetilde{\zz}^{(i)}_k$, for $k=1,\ldots, p$, $i=0,1,\ldots,$ be determined by Theorem \ref{9772n}  and \ref{2m7a}, respectively. Let $\varepsilon^{(i)}$ be the associated error defined by (\ref{t87j}) and (\ref{bb7j}). Then the increase in the number of iterations $i$ implies the decrease in the associated error, i.e.,
 \begin{eqnarray}\label{ww4v}
 \varepsilon^{(i+1)}\leq \varepsilon^{(i)}.
 \end{eqnarray}
\end{theorem}

\begin{proof} Let us  consider the initial  case of the proposed method when $i=0$.

{\sf\em (i) The case ${\mathbf {\mathit i}=0}$.} For $i=0$, the $i$-th iteration loop represented  in Section \ref{nm33}, implies
$$
\varepsilon^{(1)}_z = \min_{{\mathbf z}_1,\ldots, {\mathbf z}_p} \|\x - \sum_{k=1}^{p} G^{(0)}_k H^{(0)}_k \zz_k\|^2_\Omega.
$$
Here, for $j=1,\ldots,p,$  $G^{(1)}_j = G^{(0)}_j$ and $H^{(1)}_j = H^{(0)}_j$. Therefore,
$$
\varepsilon^{(1)}_z = \min_{{\mathbf z}_1,\ldots, {\mathbf z}_p} \|\x - \sum_{k=1}^{p} G^{(1)}_k H^{(1)}_k \zz_k\|^2_\Omega,
$$
i.e., for any ${\mathbf z}_1,\ldots, {\mathbf z}_p$,
$$
\varepsilon^{(1)}_z \leq \|\x - \sum_{k=1}^{p} G^{(1)}_k H^{(1)}_k \zz_k\|^2_\Omega.
$$
In particular, for ${\mathbf z}_1=\widetilde{\mathbf z}^{(0)}_1,\ldots, {\mathbf z}_p=\widetilde{\mathbf z}^{(0)}_p$,
$$
\varepsilon^{(1)}_z \leq \|\x - \sum_{k=1}^{p} G^{(1)}_k H^{(1)}_k \widetilde{\zz}^{(0)}_k\|^2_\Omega = \varepsilon^{(1)}_{GH}
$$
and
$$
\varepsilon^{(1)}_z \leq \|\x - \sum_{k=1}^{p} G^{(0)}_k H^{(0)}_k \widetilde{\zz}^{(0)}_k\|^2_\Omega = \varepsilon^{(0)}.
$$
Let us denote
\begin{eqnarray}\label{uy9v}
\varepsilon^{(1)} = \varepsilon^{(1)}_z.
\end{eqnarray}
Then
\begin{eqnarray}\label{kk5m}
\varepsilon^{(1)} \leq \varepsilon^{(0)}.
\end{eqnarray}

For $i=1,2,\ldots,$ let us prove inequality (\ref{ww4v}) by induction. To this end, we first consider the basis step of the induction which consists of cases $i=1$ and $i=2$.

{\sf\em (ii) The  basis step: Case ${\mathbf {\mathit i}=1}$.} If $i=1$ then the $i$-th iteration loop (see Section \ref{nm33}) implies
$$
 \varepsilon^{(2)}_{GH}=\min_{\substack{G_1, H_1,\ldots, G_p, H_p\\}}  \left\|\x -  \sum_{k=1}^{p}G_k H_k \widetilde{\zz}^{(1)}_k\right\|^2_\Omega = \left\|\x -  \sum_{k=1}^{p}G^{(2)}_k H^{(2)}_k \widetilde{\zz}^{(1)}_k\right\|^2_\Omega,
$$
i.e., for all $G_k$ and $H_k$ with $k=1,\ldots,p$,
$$
\varepsilon^{(2)}_{GH} \leq \left\|\x -  \sum_{k=1}^{p}G_k H_k \widetilde{\zz}^{(1)}_k\right\|^2_\Omega.
$$
In particular, for $G_k = G^{(1)}_k$ and $H_k = H^{(1)}_k$ with $k=1,\ldots,p$,
\begin{eqnarray}\label{mm1v}
\varepsilon^{(2)}_{GH} \leq \left\|\x -  \sum_{k=1}^{p}G^{(1)}_k H^{(1)}_k \widetilde{\zz}^{(1)}_k\right\|^2_\Omega.
\end{eqnarray}
Further, because  for $k=1,\ldots,p,$  $G^{(1)}_k = G^{(0)}_k$ and $H^{(1)}_k = H^{(0)}_k$ then
\begin{eqnarray}\label{h64v}
\min_{{\mathbf z}_1,\ldots, {\mathbf z}_p} \|\x - \sum_{k=1}^{p} G^{(1)}_k H^{(1)}_k \zz_k\|^2_\Omega = \min_{{\mathbf z}_1,\ldots, {\mathbf z}_p} \|\x - \sum_{k=1}^{p} G^{(0)}_k H^{(0)}_k \zz_k\|^2_\Omega.
\end{eqnarray}
Therefore, $\zz^{(2)}_k = \zz^{(1)}_k$, for $k=1,\ldots,p$ (see (\ref{gxa1j})). Thus, (\ref{mm1v}) and (\ref{h64v}) imply
\begin{eqnarray}\label{231v}
\varepsilon^{(2)}_{GH} \leq \left\|\x -  \sum_{k=1}^{p}G^{(1)}_k H^{(1)}_k \widetilde{\zz}^{(2)}_k\right\|^2_\Omega = \varepsilon^{(2)}_{z}.
\end{eqnarray}
But since $G^{(1)}_k = G^{(0)}_k$ and $H^{(1)}_k = H^{(0)}_k$, for $k=1,\ldots,p$, then
\begin{eqnarray}\label{88nv}
\varepsilon^{(2)}_{z} = \varepsilon^{(1)}_{z}.
\end{eqnarray}
Then by (\ref{uy9v}), (\ref{231v}) and (\ref{88nv}),
\begin{eqnarray}\label{8nnm}
\varepsilon^{(2)}_{GH} \leq  \varepsilon^{(1)}_{z} =  \varepsilon^{(1)}.
\end{eqnarray}
Because of (\ref{231v}) we denote
\begin{eqnarray}\label{897m}
\varepsilon^{(2)} = \varepsilon^{(2)}_{GH}.
\end{eqnarray}
Then (\ref{8nnm}) and (\ref{897m}) imply
\begin{eqnarray}\label{8j7m}
\varepsilon^{(2)} \leq \varepsilon^{(1)}.
\end{eqnarray}

{\sf\em (iii) The  basis step: Case ${\mathbf {\mathit i}=2}$.} In this case, $\varepsilon^{(i+1)}_z$ in (\ref{zt7j}) is written as
$$
\varepsilon^{(3)}_z = \min_{{\mathbf z}_1,\ldots, {\mathbf z}_p} \|\x - \sum_{k=1}^{p} G^{(2)}_k H^{(2)}_k \zz_k\|^2_\Omega = \|\x - \sum_{k=1}^{p} G^{(2)}_k H^{(2)}_k \widetilde{\zz}^{(3)}_k\|^2_\Omega.
$$
That is, for any $\zz_1,\ldots, \zz_p,$
$$
\varepsilon^{(3)}_z \leq \|\x - \sum_{k=1}^{p} G^{(2)}_k H^{(2)}_k {\zz}_k\|^2_\Omega.
$$
In particular, for $\zz_k = \widetilde{\zz}^{(2)}_k$ and $k=1,\ldots,p,$
$$
\varepsilon^{(3)}_z \leq \|\x - \sum_{k=1}^{p} G^{(2)}_k H^{(2)}_k \widetilde{\zz}^{(2)}_k\|^2_\Omega,
$$
where by (\ref{h64v}), $\widetilde{\zz}^{(2)}_k = \widetilde{\zz}^{(1)}_k$. Thus,
\begin{eqnarray}\label{f5m}
\varepsilon^{(3)}_z \leq \|\x - \sum_{k=1}^{p} G^{(2)}_k H^{(2)}_k \widetilde{\zz}^{(1)}_k\|^2_\Omega.
\end{eqnarray}
At the same time,
\begin{eqnarray}\label{bb7m}
 \varepsilon^{(3)}_{GH}=\min_{\substack{G_1, H_1,\ldots, G_p, H_p\\}}  \left\|\x -  \sum_{k=1}^{p}G_k H_k \widetilde{\zz}^{(2)}_k\right\|^2_\Omega = \left\|\x -  \sum_{k=1}^{p}G^{(3)}_k H^{(3)}_k \widetilde{\zz}^{(2)}_k\right\|^2_\Omega.
\end{eqnarray}
But $\widetilde{\zz}^{(2)}_k = \widetilde{\zz}^{(1)}_k$, therefore, $G^{(3)}_k, H^{(3)}_k$ solve, in fact,
$$
\min_{\substack{G_1, H_1,\ldots, G_p, H_p\\}}  \left\|\x -  \sum_{k=1}^{p}G_k H_k \widetilde{\zz}^{(1)}_k\right\|^2_\Omega,
$$
i.e., $G^{(3)}_k = G^{(2)}_k , H^{(3)}_k = H^{(2)}_k$. Therefore, (\ref{bb7m}) implies
\begin{eqnarray}\label{br37m}
 \varepsilon^{(3)}_{GH} = \left\|\x -  \sum_{k=1}^{p}G^{(2)}_k H^{(2)}_k \widetilde{\zz}^{(1)}_k\right\|^2_\Omega = \varepsilon^{(2)}_{GH}
\end{eqnarray}
and (\ref{f5m}), (\ref{br37m}) imply
\begin{eqnarray}\label{b23}
 \varepsilon^{(3)}_z\leq  \varepsilon^{(3)}_{GH} = \varepsilon^{(2)}_{GH}.
\end{eqnarray}
Denote $\varepsilon^{(3)} = \varepsilon^{(3)}_z$. By (\ref{897m}), $\varepsilon^{(2)} = \varepsilon^{(2)}_{GH}.$ Then
\begin{eqnarray}\label{op27}
 \varepsilon^{(3)} \leq \varepsilon^{(2)}.
\end{eqnarray}

{\sf\em The inductive step.} Let us suppose that, for $s=1,2,\ldots,$ if $\widetilde{\zz}^{(2s)}$ $= \widetilde{\zz}^{(2s-1)}$, and $ \varepsilon^{(2s)} = \varepsilon^{(2s)}_{GH}$ and $ \varepsilon^{(2s-1)} = \varepsilon^{(2s-1)}_{z}$ then
\begin{eqnarray}\label{yy5m}
 \varepsilon^{(2s)} \leq \varepsilon^{(2s-1)}.
\end{eqnarray}
Below, we show that then (\ref{ww4v}) is true.

To this end, for $s=1,2,\ldots,$ let us consider case $i=2s$. We have
$$
\varepsilon^{(2s+1)}_z = \min_{{\mathbf z}_1,\ldots, {\mathbf z}_p} \|\x - \sum_{k=1}^{p} G^{(2s)}_k H^{(2s)}_k \zz_k\|^2_\Omega = \|\x - \sum_{k=1}^{p} G^{(2s)}_k H^{(2s)}_k \widetilde{\zz}^{(2s+1)}_k\|^2_\Omega.
$$
That is, for any $\zz_1,\ldots, \zz_p,$
$$
\varepsilon^{(2s+1)}_z \leq \|\x - \sum_{k=1}^{p} G^{(2s)}_k H^{(2s)}_k {\zz}_k\|^2_\Omega.
$$
In particular, for $\zz_k = \widetilde{\zz}^{(2s)}_k$ and $k=1,\ldots,p,$
$$
\varepsilon^{(2s+1)}_z \leq \|\x - \sum_{k=1}^{p} G^{(2s)}_k H^{(2s)}_k \widetilde{\zz}^{(2s)}_k\|^2_\Omega,
$$
where by the above assumption, $\widetilde{\zz}^{(2s)}_k = \widetilde{\zz}^{(2s-1)}_k$. Thus,
\begin{eqnarray}\label{re5m}
\varepsilon^{(2s+1)}_z \leq \|\x - \sum_{k=1}^{p} G^{(2s)}_k H^{(2s)}_k \widetilde{\zz}^{(2s-1)}_k\|^2_\Omega = \varepsilon^{(2s)}_{GH},
\end{eqnarray}
i.e.,
\begin{eqnarray}\label{7t5m}
\varepsilon^{(2s+1)}_z \leq \varepsilon^{(2s)}_{GH}.
\end{eqnarray}
Denote
\begin{eqnarray}\label{yff5m}
\varepsilon^{(2s+1)} = \varepsilon^{(2s+1)}_{z}.
 \end{eqnarray}
 By the assumption, $ \varepsilon^{(2s)} = \varepsilon^{(2s)}_{GH}$. Then (\ref{7t5m}) implies
\begin{eqnarray}\label{wmym}
 \varepsilon^{(2s+1)} \leq \varepsilon^{(2s)}.
\end{eqnarray}

We also need the following. By (\ref{vvh7j}), $G^{(2s+1)}_k$ and $H^{(2s+1)}_k$ solve
\begin{eqnarray}\label{99h7j}
\min_{\substack{G_1, H_1,\ldots, G_p, H_p\\}}  \left\|\x -  \sum_{k=1}^{p}G_k H_k \widetilde{\zz}^{(2s)}_k\right\|^2_\Omega.
\end{eqnarray}
But by the assumption, $\widetilde{\zz}^{(2s)} = \widetilde{\zz}^{(2s-1)}$. Therefore, (\ref{99h7j}) is equivalent to
\begin{eqnarray}\label{cch7j}
\min_{\substack{G_1, H_1,\ldots, G_p, H_p\\}}  \left\|\x -  \sum_{k=1}^{p}G_k H_k \widetilde{\zz}^{(2s-1)}_k\right\|^2_\Omega.
\end{eqnarray}
Thus,
\begin{eqnarray}\label{sm27j}
G^{(2s+1)}_k = G^{(2s)}_k \qa H^{(2s+1)}_k = H^{(2s)}_k,
\end{eqnarray}
and, in particular, $ \varepsilon^{(2s+1)}_{GH} =  \varepsilon^{(2s)}_{GH}.$

Further, for $s=1,2,\ldots,$ let us now consider the case $i=2s+1$. Then
\begin{eqnarray*}\label{ttrr}
 &&\hspace*{-5mm}\varepsilon^{(2s+2)}_{GH}=\min_{\substack{G_1, H_1,\ldots, G_p, H_p\\}}  \left\|\x -  \sum_{k=1}^{p}G_k H_k \widetilde{\zz}^{(2s+1)}_k\right\|^2_\Omega \\
 &&\hspace*{40mm} = \left\|\x -  \sum_{k=1}^{p}G^{(2s+2)}_k H^{(2s+2)}_k \widetilde{\zz}^{(2s+1)}_k\right\|^2_\Omega.
\end{eqnarray*}
Therefore, for any $G_k$ and  $H_k$,
\begin{eqnarray}\label{t271r}
 \varepsilon^{(2s+2)}_{GH}\leq\left\|\x -  \sum_{k=1}^{p}G_k H_k \widetilde{\zz}^{(2s+1)}_k\right\|^2_\Omega.
\end{eqnarray}
In particular, for $G_k=G^{(2s+1)}_k$ and  $H_k=H^{(2s+1)}_k$,
\begin{eqnarray}\label{9971r}
 \varepsilon^{(2s+2)}_{GH}\leq\left\|\x -  \sum_{k=1}^{p}G^{(2s+1)}_k H^{(2s+1)}_k \widetilde{\zz}^{(2s+1)}_k\right\|^2_\Omega.
\end{eqnarray}
Then (\ref{sm27j}) implies
\begin{eqnarray}\label{8871r}
 \varepsilon^{(2s+2)}_{GH}\leq\left\|\x -  \sum_{k=1}^{p}G^{(2s)}_k H^{(2s)}_k \widetilde{\zz}^{(2s+1)}_k\right\|^2_\Omega,
\end{eqnarray}
But the RHS in (\ref{8871r}) is $\varepsilon^{(2s+1)}_z$. Denote $\varepsilon^{(2s+2)}=\varepsilon^{(2s+2)}_{GH}$ and recall that  by (\ref{yff5m}),
$\varepsilon^{(2s+1)}= \varepsilon^{(2s+1)}_{z}.$ Therefore, (\ref{8871r}) implies
\begin{eqnarray}\label{vvi9}
 \varepsilon^{(2s+2)} \leq  \varepsilon^{(2s+1)}.
\end{eqnarray}
Thus, (\ref{ww4v}) is true. $\hfill\blacksquare$
\end{proof}

Further, the following observations are a basis for Theorem \ref{v5v} considered below.
Let us denote
\begin{eqnarray}\label{aas1}
f(S,\zz) = f(S_0,\ldots,S_p,\zz_0,\ldots,\zz_p) = \|\x - \sum_{j=0}^pG_j H_j\zz_j \|^2_\Omega,
\end{eqnarray}
where, for $j=0,1,\ldots,p,$  $S_j=G_jH_j$, $S=[S_0,\ldots,S_p]$, $\zz = [\zz_0,\ldots,\zz_p]^T$ and $\zz_0=\y$.

For $i=1,2,\ldots$, we also write  $S^{(i)} = [S_0^{(1)}, S^{(i)}_1,\ldots,S^{(i)}_p]$ and $\widetilde{\zz}^{(i-1)}=[\widetilde{\zz}^{(i-1)}_0,\widetilde{\zz}^{(i-1)}_1,\ldots,\widetilde{\zz}^{(i-1)}_p]^T$ where $\widetilde{\zz}^{(i-1)}_0 = \y$.

Let $\mathcal M$ be a metric space. Following \cite{Amann2006} (p. 134), we call $\alpha \in \mathcal M$ a cluster point of  sequence $\{x_n\} \subset \mathcal M$ if every neighborhood of  $\alpha$ contains infinitely many terms of the sequence. A more precise characterization
of the  cluster point is as follows.

\begin{definition}
Let  $\mathbb{B}(\alpha,\epsilon) \subset \mathcal M$ be an open ball with center $\alpha$ and radius $\epsilon$.
 Point $\alpha$  is called  a {\em cluster point} \cite{Amann2006} of a sequence $\{x_n\} \subset \mathcal M$ iff for each $\epsilon>0$, $m\in\mathbb{N}$, there is $n\geq m$ such that $x_n\in\mathbb{B}(\alpha,\epsilon)$, for all $n\geq m$.
\end{definition}

\begin{definition}
Let $P^{*}=(S^{*},\zz^{*})$ be a cluster point of sequence $\{P^{(j)}\}=\{(S^{(j)},\zz^{(j)})\}$, for $j=1,2,\ldots.$ Point $\overline{S}_{\zz^*}$ is called a {\em best response} to $S$ if
$$
\overline{S}_{{\mathbf z}^*}\in\arg\min_{S\in K_1}f(S,\zz^*).
$$
\end{definition}

We wish now to find a  coordinate-wise minimum point of $f(S,\zz)$. To this end, we  define compact sets $K_1$ and $K_2$  such that
\begin{eqnarray}\label{k1234}
K_1=\{S:0\leq f(S,\zz^{(i-1)})\leq f(S^{(i)},\widetilde{\zz}^{(i-1)})\}
\end{eqnarray}
and
\begin{eqnarray}\label{k1235}
 K_2=\{\zz:0\leq f(S^{(i)},\zz))\leq f(S^{(i)},\widetilde{\zz}^{(i-1)})\}.
\end{eqnarray}

\begin{theorem}\label{v5v}
Let $S\in K_1$ and $\zz\in K_2.$  Let $S^{(i)}$ and $\widetilde{\zz}^{(i)}$ be determined by Theorems \ref{9772n}  and \ref{2m7a}, respectively, and let $P^{(i)} = (S^{(i)},\widetilde{\zz}^{(i-1)})$. Then any cluster point of sequence $\{P^{(i)}\}$, say $P^{*}=(S^{*},\zz^{*})$ , is a coordinate-wise minimum point of $f(S,\zz)$, i.e.,
$$
S^{*} \in \arg \min_{S\in K_1} f(S,\zz^{*}) \qa \zz^{*} \in \arg \min_{{\mathbf z}\in K_2} f(S^{*},\zz).
$$
\end{theorem}

\begin{proof}
Since each $K_j$, for $j=1,2$, is compact, then there is a subsequence $\{P^{(j_t)}\}=\{(S^{(j_t)},{\mathbf z}^{(j_t)})\}$ such that $P^{(j_t)}\rightarrow P^{*}$ when $t\rightarrow \infty$.
Consider  entry $S^{(j_t)}$ of $P^{(j_t)}$. Let $\overline{S}_{{\mathbf z}^*}$ and  $\overline{S}_{{\mathbf z}^{(j_t)}}$ be best responses to $S$ associated with ${{\mathbf z}^*}$ and ${\mathbf z}^{(j_t)}$, respectively.
Then we have \cite{Zhening2015}
\begin{eqnarray*}
\hspace*{0mm}f(\overline{S}_{{\mathbf z}^*},{\mathbf z}^{(j_t)})  \geq  f(\overline{S}_{{\mathbf z}^{(j_t)}},{\mathbf z}^{(j_t)})\geq  f(S^{(j_t+1)},{\mathbf z}^{(j_t+1)})\geq  f(S^{(j_{t+1})},{\mathbf z}^{(j_{t+1})})
\end{eqnarray*}
By continuity, as $t\rightarrow \infty$,
\begin{eqnarray}\label{29bn1}
f(\overline{S}_{{\mathbf z}^*},{\mathbf z}^{*})\geq f(S^{*},\zz^{*}).
\end{eqnarray}
 It implies that (\ref{29bn1}) should hold as an equality, since the inequality is true by the definition of the best response $\overline{S}_{{\mathbf z}^*}$. Thus, $S^{*}$ is the best response for $\zz^*$, or equivalently, $\overline{S}^*$ is the solution for the problem
$\displaystyle \arg \min_{S\in K_1} f(S,\zz^{*}).$

The proof  is similar if we consider entry $\zz^{(j_t)}$ of $P^{(j_t)}$.
 $\hfill\blacksquare$
\end{proof}


\end{document}